\documentclass{amsart}

\title{Open determinacy for class games}

\author{Victoria Gitman}
\address[V. Gitman]{The City University of New York, CUNY Graduate Center, Mathematics Program, 365 Fifth Avenue, New York, NY 10016}
\email{vgitman@nylogic.org}
\urladdr{http://boolesrings.org/victoriagitman}

\author{Joel David Hamkins}
 \address[J.~D.~Hamkins]
         {Philosophy, New York University \&
         Mathematics, Philosophy, Computer Science, The Graduate Center of The City University of New York, 365 Fifth Avenue, New York, NY 10016 \& Mathematics, College of Staten Island of CUNY}
\email{jhamkins@gc.cuny.edu}
\urladdr{http://jdh.hamkins.org}

\thanks{The authors would like to acknowledge helpful exchanges with Thomas A.~Johnstone, Stephen G.~Simpson, Philip Welch and Kentaro Fujimoto concerning the work of this paper. The second author is grateful for the support of Simons Foundation grant 209252 and for his Visiting Professorship in 2015 in the Philosophy Department of New York University, where the initial main result was obtained. Both authors are thankful for the support provided in Summer and Fall of 2015 by the Isaac Newton Institute for Mathematical Sciences in Cambridge, U.K., where the work was finalized while they were Visiting Fellows for the program on the Mathematical, Computational and Foundational Aspects of the Higher Infinite. Commentary concerning this paper can be made at \href{http://jdh.hamkins.org/open-determinacy-for-class-games}{http://jdh.hamkins.org/open-determinacy-for-class-games}.}
\thanks{We are especially pleased to be a part of this volume celebrating the 60th birthday of W.~Hugh Woodin, in light of the fact that Woodin served (years ago) as the PhD supervisor of the second author, who himself served (more recently) as the PhD supervisor of the first author. Thus, with Woodin, we span three mathematical generations.}

\usepackage{latexsym,amsfonts,amsmath,amssymb,mathrsfs}
\usepackage[hidelinks]{hyperref}
\usepackage{diagrams}
\usepackage{verbatim}

\diagramstyle[tight,centredisplay]
\newtheorem{theorem}{Theorem}

\newtheorem*{maintheorem*}{Main Theorem}
\newtheorem*{maintheorems*}{Main Theorems}

\newtheorem*{corollary*}{Corollary}
\newtheorem*{corollaries*}{Corollaries}
\newtheorem{sublemma}{Lemma}[theorem]
\newtheorem{lemma}[theorem]{Lemma}

\newtheorem{question}[theorem]{Question}

\newtheorem*{questions*}{Questions}
\newtheorem*{mainquestion*}{Main Question} 
\newtheorem*{openquestion*}{Open Question} 

\newtheorem{definition}[theorem]{Definition}

\newcommand{\QED}{\end{proof}}

\def\proclaim[#1]{{\bf #1}}
\def\BF#1.{{\bf #1.}}

\newcommand{\Godel}{G\"odel}

\newcommand{\Levy}{L\'{e}vy}

\newcommand{\of}{\subseteq}

\newcommand{\set}[1]{\{\,{#1}\,\}}

\newcommand{\elesub}{\prec}

\newcommand{\Con}{\mathop{{\rm Con}}}
\newcommand{\image}{\mathbin{\hbox{\tt\char'42}}}

\newcommand{\restrict}{\upharpoonright} 
\newcommand{\satisfies}{\models}

\newcommand{\concat}{\mathbin{{}^\smallfrown}}

\newcommand{\Union}{\bigcup}

\newcommand{\intersect}{\cap}

\newcommand{\trianglelt}{\lhd}

\newcommand{\smalllt}{\mathrel{\mathchoice{\raise2pt\hbox{$\scriptstyle<$}}{\raise1pt\hbox{$\scriptstyle<$}}{\raise0pt\hbox{$\scriptscriptstyle<$}}{\scriptscriptstyle<}}}
\newcommand{\smallleq}{\mathrel{\mathchoice{\raise2pt\hbox{$\scriptstyle\leq$}}{\raise1pt\hbox{$\scriptstyle\leq$}}{\raise1pt\hbox{$\scriptscriptstyle\leq$}}{\scriptscriptstyle\leq}}}

\newcommand{\ltomega}{{{\smalllt}\omega}}

\newcommand{\boolval}[1]{\mathopen{\lbrack\!\lbrack}\,#1\,\mathclose{\rbrack\!\rbrack}}
\def\[#1]{\boolval{#1}}
\newbox\gnBoxA
\newdimen\gnCornerHgt
\setbox\gnBoxA=\hbox{\tiny$\ulcorner$}
\global\gnCornerHgt=\ht\gnBoxA
\newdimen\gnArgHgt
\def\gcode #1{%
\setbox\gnBoxA=\hbox{$#1$}%
\gnArgHgt=\ht\gnBoxA%
\ifnum     \gnArgHgt<\gnCornerHgt \gnArgHgt=0pt%
\else \advance \gnArgHgt by -\gnCornerHgt%
\fi \raise\gnArgHgt\hbox{\tiny$\ulcorner$} \box\gnBoxA %
\raise\gnArgHgt\hbox{\tiny$\urcorner$}}
\newcommand{\UnderTilde}[1]{{\setbox1=\hbox{$#1$}\baselineskip=0pt\vtop{\hbox{$#1$}\hbox to\wd1{\hfil$\sim$\hfil}}}{}}
\newcommand{\Undertilde}[1]{{\setbox1=\hbox{$#1$}\baselineskip=0pt\vtop{\hbox{$#1$}\hbox to\wd1{\hfil$\scriptstyle\sim$\hfil}}}{}}
\newcommand{\undertilde}[1]{{\setbox1=\hbox{$#1$}\baselineskip=0pt\vtop{\hbox{$#1$}\hbox to\wd1{\hfil$\scriptscriptstyle\sim$\hfil}}}{}}
\newcommand{\UnderdTilde}[1]{{\setbox1=\hbox{$#1$}\baselineskip=0pt\vtop{\hbox{$#1$}\hbox to\wd1{\hfil$\approx$\hfil}}}{}}
\newcommand{\Underdtilde}[1]{{\setbox1=\hbox{$#1$}\baselineskip=0pt\vtop{\hbox{$#1$}\hbox to\wd1{\hfil\scriptsize$\approx$\hfil}}}{}}

\renewcommand{\implies}{\mathrel{\rightarrow}}

\renewcommand{\iff}{\mathrel{\leftrightarrow}}

\def\<#1>{\left\langle#1\right\rangle}

\newcommand{\Ord}{\mathord{{\rm Ord}}}

\newcommand{\ETR}{{\rm ETR}}
\newcommand\ACA{{\rm ACA}}
\newcommand\RCA{{\rm RCA}}
\newcommand{\ATR}{{\rm ATR}}

\newcommand{\ZFC}{{\rm ZFC}}
\newcommand{\ZF}{{\rm ZF}}
\newcommand{\ZFCm}{\ZFC^-}
\newcommand{\ZFCmm}{\ZFC\text{\tt -}}

\newcommand{\KM}{{\rm KM}}
\newcommand{\GB}{{\rm GB}}
\newcommand{\GBC}{{\rm GBC}}

\newcommand{\DC}{{\rm DC}}

\newcommand{\RA}{{\rm RA}}

\newcommand{\PA}{{\rm PA}}

\newcommand{\cell}[1]{\boxit{\hbox to 17pt{\strut\hfil$#1$\hfil}}}
\newcommand{\head}[2]{\lower2pt\vbox{\hbox{\strut\footnotesize\it\hskip3pt#2}\boxit{\cell#1}}}
\newcommand{\boxit}[1]{\setbox4=\hbox{\kern2pt#1\kern2pt}\hbox{\vrule\vbox{\hrule\kern2pt\box4\kern2pt\hrule}\vrule}}
\newcommand{\Col}[3]{\hbox{\vbox{\baselineskip=0pt\parskip=0pt\cell#1\cell#2\cell#3}}}
\newcommand{\tapenames}{\raise 5pt\vbox to .7in{\hbox to .8in{\it\hfill input: \strut}\vfill\hbox to
.8in{\it\hfill scratch: \strut}\vfill\hbox to .8in{\it\hfill output: \strut}}}
\newcommand{\Head}[4]{\lower2pt\vbox{\hbox to25pt{\strut\footnotesize\it\hfill#4\hfill}\boxit{\Col#1#2#3}}}
\newcommand{\Dots}{\raise 5pt\vbox to .7in{\hbox{\ $\cdots$\strut}\vfill\hbox{\ $\cdots$\strut}\vfill\hbox{\
$\cdots$\strut}}}
\newcommand{\df}{\it} 
\newcommand\Tr{\text{Tr}}

\begin{document}

\begin{abstract}
 The principle of open determinacy for class games---two-player games of perfect information with plays of length $\omega$, where the moves are chosen from a possibly proper class, such as games on the ordinals---is not provable in Zermelo-Fraenkel set theory \ZFC\ or \Godel-Bernays set theory \GBC, if these theories are consistent, because provably in \ZFC\ there is a definable open proper class game with no definable winning strategy. In fact, the principle of open determinacy and even merely clopen determinacy for class games implies $\Con(\ZFC)$ and iterated instances $\Con^\alpha(\ZFC)$ and more, because it implies that there is a satisfaction class for first-order truth, and indeed a transfinite tower of truth predicates $\Tr_\alpha$ for iterated truth-about-truth, relative to any class parameter. This is perhaps explained, in light of the Tarskian recursive definition of truth, by the more general fact that the principle of clopen determinacy is exactly equivalent over \GBC\ to the principle of elementary transfinite recursion \ETR\ over well-founded class relations. Meanwhile, the principle of open determinacy for class games is provable in the stronger theory $\GBC+\Pi^1_1$-comprehension, a proper fragment of Kelley-Morse set theory \KM.
\end{abstract}

\maketitle

\section{Introduction}

\noindent The past half-century of set theory has revealed a robust connection between infinitary game theory and fundamental set-theoretic principles, such as the existence of certain large cardinals. The existence of strategies in infinite games has often turned out to have an unexpected set-theoretic power. In this article, we should like to exhibit another such connection in the case of games of proper class size, by proving that the principle of clopen determinacy for class games is exactly equivalent to the principle of elementary transfinite recursion \ETR\ along well-founded class relations. Since this principle implies $\Con(\ZFC)$ and iterated instances of $\Con^\alpha(\ZFC)$ and more, the principles of open determinacy and clopen determinacy both transcend \ZFC\ in consistency strength.

We consider two-player games of perfect information, where two players alternately play elements from an allowed space $X$ of possible moves, which in our case may be a proper class, such as the class of all ordinals $X=\Ord$. Together, the players build an infinite sequence $\vec\alpha=\langle \alpha_0,\alpha_1,\alpha_2,\ldots\rangle$ in $X^\omega$, which is the play resulting from this particular instance of the game. The winner is determined by consulting a fixed class of plays $A\of X^\omega$, possibly a proper class: if $\vec\alpha\in A$, then the first player has won this play of the game, and otherwise the second player has won. A strategy for a player is a (class) function $\sigma:X^{\ltomega}\to X$, which tells a player how to move next, given a finite position in the game. Such a strategy is winning for that player, if playing in accordance with the strategy leads to a winning play of the game, regardless of how the other player has moved. The game is determined, if one of the players has a winning strategy. We may formalize all talk of classes here in \Godel-Bernays \GBC\ set theory, or in \ZFC\ if one prefers to regard classes as definable from parameters.

The case of open games, generalizing the finite games, is an attractive special case, which for set-sized games has been useful in many arguments. Specifically, a game is {\df open} for a particular player, if for every winning play of the game for that player, there occurred during the course of play a finite position where the winning outcome was already ensured, in the sense that all plays extending that position are winning for that player. This is equivalent to saying that the winning condition set for that player is open in the product topology on $X^\omega$, where we put the discrete topology on $X$. Similarly, a game is {\df clopen}, if it is open for each player; these are the games for which every play of the game has a finite stage where the outcome is already known.

It is a remarkable elementary fact, the Gale-Stewart theorem~\cite{GaleStewart1953:InfiniteGamesWithPerfectInformation}, that in the context of set-sized games, every open game is determined. In order to discuss the problems that arise in the context of proper-class games, let us briefly sketch two classic proofs of open determinacy for set games. Suppose that we have a game that is open for one of the players, with an open winning condition $A\of X^\omega$ for that player, where $X$ is the set of possible moves.

For the first proof of open determinacy, suppose that the open player does not have a winning strategy in the game. So the initial position of the game is amongst the set of positions from which the open player does not have a winning strategy. The closed player may now simply play so as to stay in that collection, because if every move from a position $p$ leads to a winning position for the open player, then the open player can unify those strategies into a single winning strategy from position $p$. This way of playing is a winning strategy for the closed player, because no such position can be an already-won position for the open player. Therefore, the game is determined.

For a second proof, we use the elegant theory of ordinal game values. Namely, define that a position $p$ in the game has value $0$, if it is an already-won position for the open player, in the sense that every play extending $p$ is in $A$. A position $p$ with the open player to move has value $\alpha+1$, if $\alpha$ is minimal such that the open player can play to a position $p\concat x$ with value $\alpha$. A position $p$ with the closed player to play has a defined value only when every possible subsequent position $p\concat y$ already has a value, and in this case the value of $p$ is the supremum of those values. The key observation is that if a position has a value, then the open player can play so as to decrease the value, and the closed player cannot play so as to increase it or make it become undefined. Thus, by means of this value-reducing strategy, the open player can win from any position having a value, because the decreasing sequence of ordinals must eventually hit $0$; and the closed player can win from any position lacking an ordinal value, by maintaining the play on unvalued positions. So the game is determined, because the initial position either has a value or is unvalued. (The topic of ordinal game values appears widely in the literature, but for a particularly accessible discussion of the concrete meaning of small-ordinal game values, we refer the reader to~\cite{EvansHamkins2014:TransfiniteGameValuesInInfiniteChess, EvansHamkinsPerlmutter:APositionInInfiniteChessWithGameValueOmega^4}.)

Both of these proofs of open determinacy become problematic in \ZFC\ and also \GBC\ when $X$ is a proper class. The problem with the first proof is subtle, but note that winning strategies are proper class functions $\sigma:X^\ltomega\to X$, and so for a position $p$ to have a winning strategy for a particular player is a second-order property of that position. Thus, in order to pick out the class of positions for which the open player has a winning strategy, we would seem to need a second-order comprehension principle, which is not available in \ZFC\ or \GBC. The proof can, however, be carried out in $\GBC+\Pi^1_1$-comprehension, as explained in theorem~\ref{Theorem.KMImpliesOpenDeterminacy}. The problem with the ordinal-game-value proof in the proper class context is a little more clear, since in the inductive definition of game values, one takes a supremum of the values of $p\concat x$ for all $x\in X$, but if $X$ is a proper class, this supremum could exceed every ordinal. It would seem that we would be pushed to consider meta-ordinal game values larger than $\Ord$. So the proofs just don't seem to work in \ZFC\ or \GBC.

What we should like to do in this article is to consider more seriously the case where $X$ is a proper class. In this case, as we mentioned, the strategies $\sigma:X^{\ltomega}\to X$ will also be proper classes, and the winning condition $A\of X^\omega$ may also be a proper class.

\begin{question}\label{Question.CanWeProveOpenClassDeterminacy?}
 Can we prove open determinacy for class games?
\end{question}

\noindent For example, does every definable open class game in \ZFC\ admit a definable winning strategy for one of the players? In \GBC, must every open class game have a winning strategy? We shall prove that the answers to both of these questions is no. Our main results are the following.

\begin{maintheorems*}\
 \begin{enumerate}
   \item In \ZFC, there is a first-order definable clopen proper-class game with no definable winning strategy for either player.
   \item In \GBC, the existence of a winning strategy for one of the players in the game of statement (1) is equivalent to the existence of a satisfaction class for first-order set-theoretic truth.
   \item Consequently, the principle of clopen determinacy for class games in \GBC\ implies $\Con(\ZFC)$ and iterated consistency assertions $\Con^\alpha(\ZFC)$ and more.
   \item Indeed, the principle of clopen determinacy for class games is equivalent over \GBC\ to the principle \ETR\ of elementary transfinite recursion, which is a strictly weaker theory (assuming consistency) than $\GBC+\Pi^1_1$-comprehension, which is strictly weaker than Kelley-Morse \KM\ set theory.
   \item Meanwhile, open determinacy for class games is provable in $\GBC+\Pi^1_1$-comprehension.
 \end{enumerate}
\end{maintheorems*}

These claims will be proved in theorems~\ref{Theorem.DefinableClopenGameNoDefinableWS},~\ref{Theorem.ClopenDeterminacyGivesTruthPredicate},~\ref{Theorem.ClopenDetIffETR} and~\ref{Theorem.KMImpliesOpenDeterminacy}. Note that because \GBC\ includes the global choice principle, every proper class $X$ is bijective with the class of all ordinals $\Ord$, and so in \GBC\ one may view every class game as a game on the ordinals. We shall also prove other theorems that place \ETR\ and hence clopen determinacy into their setting in second-order set theory beyond \GBC.

\section{A broader context}

Let us place the results of this article in a broader context. The need to consider clopen determinacy for class games has arisen in recent developments in large cardinal set theory and forcing axioms, which sparked our interest. Specifically, Audrito and Viale~\cite{AudritoViale:AbsolutenessViaReflection, Audrito2016:Generic-large-cardinals-and-absoluteness}, generalizing the uplifting cardinals and resurrection axioms of~\cite{HamkinsJohnstone2014:ResurrectionAxiomsAndUpliftingCardinals, HamkinsJohnstone:StronglyUpliftingCardinalsAndBoldfaceResurrection}, introduced the $(\alpha)$-uplifting cardinals, the $\text{HJ}(\alpha)$-uplifting cardinals, and the iterated resurrection axioms $\RA_\alpha(\Gamma)$, all of which are defined in terms of winning strategies in certain proper-class-sized clopen games. Audrito and Viale had at first formalized their concepts in Kelley-Morse set theory, which is able to prove the requisite determinacy principles. The main results of this paper identify the minimal extensions of \Godel-Bernays set theory able satisfactorily to treat these new large cardinals and forcing axioms.

After having proved our theorems, however, we noticed the connection with analogous results in second-order arithmetic, where there has been a vigorous investigation of the strength of determinacy over very weak theories. Let us briefly survey some of that work. First, there is a natural affinity between our theorem~\ref{Theorem.ClopenDetIffETR}, which shows that clopen determinacy is exactly equivalent over \GBC\ to the principle of elementary transfinite recursion \ETR\ over well-founded class relations, with the 1977 dissertation result of Steel (see~\cite[Thm~V.8.7]{Simpson2009:SubsystemsOfSecondOrderArithmetic}), showing that clopen determinacy for games on the natural numbers is exactly equivalent in reverse mathematics to the theory of arithmetical transfinite recursion $\text{ATR}_0$. Simpson reported in conversation with the second author that Steel's theorem had had a strong influence on the beginnings of the reverse mathematics program. In both the set and class contexts, we have equivalence of clopen determinacy with a principle of first-order transfinite recursion. In the case of games on the natural numbers, however, Steel proved that $\ATR_0$ is also equivalent with open determinacy, and not merely clopen determinacy, whereas the corresponding situation of open determinacy for class games is not yet completely settled; the best current upper bound provided by theorem~\ref{Theorem.KMImpliesOpenDeterminacy}.

After Steel, the reverse mathematics program proceeded to consider the strength of determinacy for games having higher levels of complexity. Tanaka~\cite{Tanaka1990:WeakAxiomsOfDeterminacyAndSubsystemsOfAnalysisI} established the equivalence of $\Pi^1_1$-comprehension and $\Sigma^0_1\wedge \Pi^0_1$-determinacy, as well as the equivalence of $\Pi^1_1$-transfinite recursion and $\Delta^0_2$-determinacy, both over $\RCA_0$. The subsequent paper~\cite{Tanaka1991:WeakAxiomsOfDeterminacyAndSubsystemsOfAnalysisII} showed that $\Sigma^0_2$-determinacy is equivalent over $\RCA_0$ to a less familiar second-order axiom $\Sigma^1_1$-${\rm MI}$, known as the axiom of $\Sigma^1_1$-\emph{monotone inductive definition}\footnote{A function $\Gamma:P(\omega)\to P(\omega)$ is called a monotone operator (over $\omega$) if whenever $X\subseteq Y$, then $\Gamma(X)\subseteq \Gamma(Y)$. The axiom of $\Sigma^1_1$-\emph{monotone inductive definition} asserts that for every $\Sigma_1^1$-monotone operator $\Gamma$ (meaning $\{(x,X)\mid x\in \Gamma(X)\}$ is $\Sigma_1^1$), there exists a sequence $\< \Gamma_\alpha\mid\alpha\leq\sigma>$ for some ordinal $\sigma$ such that $\Gamma_\alpha=\Gamma(\Union_{\beta<\alpha}\Gamma_\beta)$ for all $\alpha\leq\sigma$ and such that $\Gamma_\sigma=\Union_{\alpha<\sigma}\Gamma_\alpha$, so that $\Gamma_\sigma$ is a fixed point of the operator $\Gamma$.}. MedSalem and Tanaka~\cite{MedSalemTanaka2007:Delta03DeterminacyComprehensionAndInduction} considered $\Delta^0_3$-determinacy, proving it in $\Delta^1_3$-comprehension plus $\Sigma^1_3$-induction, and showing that it does not follow from $\Delta^1_3$-comprehension alone. MedSalem and Tanaka~\cite{MedSalemTanaka2008:WeakDeterminacyAndIterationsOfInductiveDefinitions} settled the exact strength of $\Delta^0_3$-determinacy over the theory $\RCA_0+\Pi^1_3$-transfinite induction by introducing a new axiom for iterating $\Sigma^1_1$-inductive definitions. Philip Welch~\cite{Welch2011:WeakSystemsOfDeterminacyAndArithmeticalQuasiInductiveDefinitions} characterized the ordinal stage by which the strategies for $\Sigma^0_3$ games appear in the constructible hierarchy, continuing the program initiated by Blass~\cite{Blass1972:ComplexityOfWinningStrategies}, who showed that every computable game has its strategy appearing before the next admissible ordinal. He also shows that $\Pi^1_3$-comprehension proves not just $\Pi^0_3$-determinacy, but that there is a $\beta$-model of $\Pi^0_3$-determinacy. Montalb{\'a}n and Shore~\cite{MontalbanShore2012:The-limits-of-determinacy-in-second-order-arithmetic} established a precise bound for the amount of determinacy provable in full second-order arithmetic ${\rm Z}_2$. They showed that for each fixed $n$, $\Pi^1_{n+2}$-comprehension proves determinacy for $n$-length Boolean combinations of $\Pi^0_4$-formulas, but ${\rm Z}_2$ cannot prove $\Delta^0_4$-determinacy.

Although there is a clear analogy between our theorems concerning clopen determinacy for proper-class games and the analysis of clopen determinacy on the natural numbers, nevertheless, one should not naively expect a tight connection between the determinacy of $\Sigma^0_n$ definable games in second-order arithmetic, say, with that of $\Sigma^0_n$ definable class games in the \Levy\ hierarchy. The reason is that determinacy for first-order definable $\Sigma^0_n$ sets of reals in the arithmetic hierarchy is provable in \ZFC, and with sufficient large cardinals, determinacy runs through the second-order projective hierarchy $\Sigma^1_n$ as well, but determinacy for first-order definable games in set theory is simply refutable in \ZF\ already at the level of $\Delta^0_2$ in the \Levy\ hierarchy, in light of theorem~\ref{Theorem.DefinableNondeterminedGame}. Rather, one should expect a connection between the analysis of $\Sigma^0_n$ determinacy in arithmetic and the corresponding level of the proper-class analogue of the Borel hierarchy for subclasses of $\Ord^\omega$, which we discuss in section~\ref{Section.Questions}. There are several fundamental disanalogies for determinacy in second-order arithmetic in comparison with second-order set theory that lead us to expect differences in the resulting theory, among them the facts that (i) $\Ord^\omega$ is not separable in the product topology whereas Baire space $\omega^\omega$ is separable; (ii) wellfoundedness for class relations is first-order expressible in set theory, whereas it is $\Pi^1_1$-complete in arithmetic; and finally, (iii) individual plays of a game on $\Ord$ are first-order objects in set theory, making the payoff collection a class, whereas in arithmetic a play of a game is already a second-order object and the payoff collection is a third-order object.

\section{The truth-telling game}

Let us now prove the initial claims of the main theorem.

\begin{theorem}\label{Theorem.DefinableClopenGameNoDefinableWS}
 In \ZFC, there is a particular definable clopen proper-class game, for which no definition and parameter defines a winning strategy for either player.
\end{theorem}

The proof therefore provides in \ZFC\ a completely uniform counterexample to clopen determinacy, with respect to definable strategies, because the particular game we shall define has no definable winning strategy for either player in any model of set theory. Theorem~\ref{Theorem.DefinableClopenGameNoDefinableWS} is a theorem scheme, ranging over the possible definitions of the putative winning strategy. We shall prove the theorem as a consequence of the following stronger and more revealing result.

\begin{theorem}\label{Theorem.ClopenDeterminacyGivesTruthPredicate}
 There is a particular first-order definable clopen game, whose determinacy is equivalent in \GBC\ to the existence of a satisfaction class for first-order set-theoretic truth. Consequently, in \GBC\ the principle of clopen determinacy for class games implies $\Con(\ZFC)$, as well as iterated consistency assertions $\Con^\alpha(\ZFC)$ and much more.
\end{theorem}

\begin{proof}
To begin, we introduce the \emph{truth-telling} game, which will be a definable open game with no definable winning strategy. The truth-telling game has two players, the {\df interrogator} and the {\df truth-teller}, who we may imagine play out the game in a court of law, with the truth-teller in the witness box answering tricky pointed questions posed by the opposing counsel, in the style of a similar game described by Adrian Mathias~\cite{Mathias2015:InLodeDellalogica} in the context of extensions of \PA\ in arithmetic. On each turn, the interrogator puts an inquiry to the the truth-teller concerning the truth of a particular first-order set-theoretic formula $\varphi(\vec a)$ with parameters. The truth-teller must reply to the inquiry by making a truth pronouncement either that it is \emph{true} or that it is \emph{false}, not necessarily truthfully, and in the case that the formula $\varphi$ is an existential assertion $\exists x\,\psi(x,\vec a)$ declared to be true, then the truth teller must additionally identify a particular witness $b$ and pronounce also that $\psi(b,\vec a)$ is \emph{true}. So a play of the game consists of a sequence of such inquiries and truth pronouncements.

The truth-teller need not necessarily answer truthfully to win! Rather, the truth-teller wins a play of the game, provided merely that she does not violate the recursive Tarskian truth conditions during the course of play. What we mean, first, is that when faced with an atomic formula, she must pronounce it true or false in accordance with the actual truth or falsity of that atomic formula; similarly, she must pronounce that $\varphi\wedge\psi$ is true just in case she pronounces both $\varphi$ and $\psi$ separately to be true, if those inquiries had been issued by the interrogator during play; she must pronounce opposite truth values for $\varphi$ and $\neg\varphi$, if both are inquired about; and she must pronounce $\exists x\,\varphi(x,\vec a)$ to be true if and only if she ever pronounces $\varphi(b,\vec a)$ to be true of any particular $b$ (the forward implication of this is already ensured by the extra pronouncement in the existential case of the game). This is an open game for the interrogator, because if the truth-teller ever should violate the Tarskian conditions, then this violation will be revealed at finite stage of play, and this is the only way for the interrogator to win.

We remind the reader that a {\df satisfaction class} or {\df truth predicate} for first-order truth is a class $\Tr$ of pairs $\<\varphi,\vec a>$ consisting of a formula $\varphi$ and a list of parameters $\vec a$ assigned to the free variables of that formula, which obeys the Tarskian recursive definition of truth (for simplicity we shall write the pair simply as $\varphi(\vec a)$, suppressing the variable assignment, but keep in mind that these are mentions of formulas rather than uses). So in the atomic case, we'll have $(a=b)\in\Tr$ if and only if $a=b$, and $(a\in b)\in\Tr$ if and only if $a\in b$; for negation, $\neg\varphi(\vec a)\in\Tr$ if and only if $\varphi(\vec a)\notin\Tr$; for conjunction, $(\varphi\wedge\psi)(\vec a)\in\Tr$ if and only if $\varphi(\vec a)\in\Tr$ and $\psi(\vec a)\in\Tr$; and for quantifiers, $\exists x\,\varphi(x,\vec a)\in\Tr$ just in case there is $b$ for which $\varphi(b,\vec a)\in \Tr$. Tarski proved that in any sufficiently strong first-order theory no such truth predicate for first-order truth is definable in the same language. Meanwhile, in the second-order Kelley-Morse set theory \KM\ and even in the weaker theory \GBC\ plus the principle of transfinite recursion over well-founded class relations, we can define a truth predicate for first-order truth, simply because the Tarskian recursion itself is a well-founded recursion on the complexity of the formulas, where we define the truth of $\varphi(\vec a)$ in terms of $\psi(\vec b)$ for simpler formulas $\psi$.

\begin{sublemma}\label{Lemma.TruthTellerStrategyIffSatisfactionClass}
 The truth-teller has a winning strategy in the truth-telling game if and only if there is a satisfaction class for first-order truth.
\end{sublemma}

\begin{proof}
We may understand this lemma as formalized in \Godel-Bernays \GBC\ set theory, which includes the global choice principle. Clearly, if there is a satisfaction class for first-order truth, then the truth-teller has a winning strategy, which is simply to answer all questions about truth in accordance with that satisfaction class, using the global choice principle to pick Skolem witnesses in the existential case. Since by definition that class obeys the Tarskian conditions, she will win the game, no matter which challenges are issued by the interrogator.

Conversely, suppose that the truth-teller has a winning strategy $\tau$ in the game. We shall use $\tau$ to build a satisfaction class for first-order truth. Specifically, let $\Tr$ be the collection of formulas $\varphi(\vec a)$ that are pronounced true by $\tau$ in any play according to $\tau$, including the supplemental truth pronouncements made in the existential case about the particular witnesses. We claim that $\Tr$ is a satisfaction class. Since the truth-teller was required to answer truthfully to all inquiries about atomic formulas, it follows that $\Tr$ contains all and only the truthful atomic assertions. In particular, the answers provided by the strategy $\tau$ on inquiries about atomic formulas are independent of the particular challenges issued by the interrogator and of the order in which they are issued. Next, we generalize this to all formulas, arguing by induction on formulas that the truth pronouncements made by $\tau$ on a formula is always independent of the play in which that formula arises. We have already noticed this for atomic formulas. In the case of negation, if inductively all plays in which $\varphi(\vec a)$ is issued as a challenge or arises as a witness case come out true, then all plays in which $\neg\varphi(\vec a)$ arises will result in false, or else we could create a play in which $\tau$ would violate the Tarskian truth conditions, simply by asking about $\varphi(\vec a)$ after $\neg\varphi(\vec a)$ was answered affirmatively. Similarly, if $\varphi$ and $\psi$ always come out the same way, then so must $\varphi\wedge\psi$. We don't claim that $\tau$ must always issue the same witness $b$ for an existential $\exists x\,\psi(x,\vec a)$, but if the strategy ever directs the truth-teller to pronounce this statement to be true, then it will provide some witness $b$ and pronounce $\psi(b,\vec a)$ to be true, and by induction this truth pronouncement for $\psi(b,\vec a)$ is independent of the play on which it arises, forcing $\exists x\,\varphi(x,\vec a)$ to always be pronounced true. Thus, by induction on formulas, the truth pronouncements made by the truth-teller strategy $\tau$ allow us to define from $\tau$ a satisfaction class for first-order truth.
\end{proof}

It follows by Tarski's theorem on the non-definability of truth that there can be no definable winning strategy for the truth-teller in this game, because there can be no definable satisfaction class.

\begin{sublemma}\label{Lemma.InterrogatorHasNoWinningStrategy}
 The interrogator has no winning strategy in the truth-telling game.
\end{sublemma}

\begin{proof}
Suppose that $\sigma$ is a strategy for the interrogator. So $\sigma$ is a proper class function that directs the interrogator to issue certain challenges, given the finite sequence of previous challenges and truth-telling answers. By the reflection theorem, there is a closed unbounded proper class of cardinals $\theta$, such that $\sigma\image V_\theta\of V_\theta$. That is, $V_\theta$ is closed under $\sigma$, in the sense that if all previous challenges and responses come from $V_\theta$, then the next challenge will also come from $V_\theta$. Since $\langle V_\theta,{\in}\rangle$ is a set, we have a truth predicate on it, as well as a Skolem function selecting existential witnesses. Consider the play, where the truth-teller replies to all inquiries by consulting truth in $V_\theta$, rather than truth in $V$, and using the Skolem function to provide the witnesses in the existential case. The point is that if the interrogator follows $\sigma$, then all the inquiries will involve only parameters $\vec a$ in $V_\theta$, provided that the truth-teller also always gives witnesses in $V_\theta$, which in this particular play will be the case. Since the truth predicate on $V_\theta$ does satisfy the Tarskian truth conditions, it follows that the truth-teller will win this instance of the game, and so $\sigma$ is not a winning strategy for the interrogator.
\end{proof}

Thus, if open determinacy holds for classes, then there is a truth predicate $\Tr$ for first-order truth. But we have not yet quite proved the theorem, because the truth-telling game is an open game, rather than a clopen game, whereas the theorem concerns determinacy for clopen games. The truth-teller wins the truth-telling game only by playing the game out for infinitely many steps, and this is not an open winning condition for her, since at any point the play could have continued in such a way so as to produce a loss for the truth-teller, if the players cooperated in order to achieve that.

So let us describe a modified game, the {\df counting-down truth-telling game}, which will be clopen and which we may use in order to prove the theorem. Specifically, the counting-down truth-telling game is just like the truth-telling game, except that we insist that the interrogator must also state on each move a specific ordinal $\alpha_n$, which descend during play $\alpha_0>\alpha_1>\cdots>\alpha_n$. If the interrogator gets to $0$, then the truth-teller is declared the winner. For this modified game, the winner will be known in finitely many moves, because either the truth-teller will violate the Tarskian conditions or the interrogator will hit zero. So this is a clopen game. Since the counting-down version of the game is harder for the interrogator, it follows that the interrogator still can have no winning strategy. We modify the proof of lemma~\ref{Lemma.TruthTellerStrategyIffSatisfactionClass} for this game by claiming that if $\tau$ is a winning strategy for the truth-teller in the counting-down truth-telling game, then the truth pronouncements made by $\tau$ in response to all plays \emph{with sufficiently large ordinals} all agree with one another independently of the interrogator's play. The inductive argument of lemma~\ref{Lemma.TruthTellerStrategyIffSatisfactionClass} still works under the assumption that the counting-down ordinal is sufficiently large, because there will be enough time to reduce a problematic case. The ordinal will depend only on the formula and not on the parameter. For example, if $\varphi(\vec a)$ always gets the same truth pronouncement for plays in which it arises with sufficiently large ordinals, then so also does $\neg\varphi(\vec a)$, with a slightly larger ordinal, because in a play with the wrong value for $\neg\varphi(\vec a)$ we may direct the interrogator to inquire next about $\varphi(\vec a)$ and get a violation of the Tarskian recursion. Similar reasoning works in the other cases, and so we may define a satisfaction class from a strategy in the modified game. Since that game is clopen, we have proved that clopen determinacy for class games implies the existence of a satisfaction class for first-order truth.

We complete the proof of theorem~\ref{Theorem.ClopenDeterminacyGivesTruthPredicate} by explaining how the existence of a satisfaction class implies $\Con(\ZFC)$ and more. Working in \Godel-Bernays set theory, we may apply the reflection theorem to the class $\Tr$ and thereby find a proper class club $C$ of cardinals $\theta$ for which $\<V_\theta,{\in},\Tr\cap  V_\theta>\elesub_{\Sigma_1}\<V,{\in},\Tr>$. In particular, this implies that $\Tr\intersect V_\theta$ is a satisfaction class on $V_\theta$, which therefore agrees with truth in that structure, and so these models form a continuous elementary chain, whose union is the entire universe:
 $$V_{\theta_0}\elesub V_{\theta_1}\elesub\cdots\elesub V_\lambda\elesub\cdots\elesub V.$$
There is a subtle point here concerning $\omega$-nonstandard models, namely, in order to see that all instances of the \ZFC\ axioms are declared true by $\Tr$, it is inadequate merely to note that we have assumed \ZFC\ to be true in $V$, because this will give us only the standard-finite instances of those axioms in $\Tr$, but perhaps we have nonstandard natural numbers in $V$, beyond the natural numbers of our metatheory. Nevertheless, because in \GBC\ we have the collection axiom relative to the truth predicate itself, we may verify that all instances of the collection axiom (including nonstandard instances, if any) $$\forall b\,\forall z\left(\strut\forall x\in b\, \exists y\,\varphi(x,y,z)\implies\exists c\,\forall x\in b\,\exists y\in c\,\varphi(x,y,z)\right)$$
must be declared true by $\Tr$, because we may replace the assertion of $\varphi(x,y,z)$ with the assertion $\varphi(x,y,z)\in\Tr$, which reduces the instance of collection for the (possibly nonstandard) formula $\varphi$ to an instance of standard-finite collection in the language of $\Tr$, using the \Godel\ code of $\varphi$ as a parameter, thereby collecting sufficient witnesses $y$ into a set $c$. So even the nonstandard instances of the collection axiom must be declared true by $\Tr$. It follows that each of these models $V_\theta$ for $\theta\in C$ is a transitive model of \ZFC, understood in the object theory of $V$, and so we may deduce $\Con(\ZFC)$ and $\Con\left(\strut\ZFC+\Con(\ZFC)\right)$ and numerous iterated consistency statements of the form $\Con^\alpha(\ZFC)$, which must be true in all such transitive models for quite a long way. Alternatively, one can make a purely syntactic argument for $\Con(\ZFC)$ from a satisfaction class, using the fact that the satisfaction class is closed under deduction and does not assert contradictions.
\end{proof}

Note that in the truth-telling games, we didn't really need the interrogator to count down in the ordinals, since it would in fact have sufficed to have him count down merely in the natural numbers; the amount of time remaining required for the truth pronouncements to stabilize is essentially related to the syntactic complexity of $\varphi$. We could have insisted merely that on the first move, the interrogator announce a natural number $n$, and then the game ends after $n$ moves, with the interrogator winning only if the Tarski conditions are violated by the truth-teller within those moves.

\begin{proof}[Proof of theorem~\ref{Theorem.DefinableClopenGameNoDefinableWS}] We use the same game as in the proof of theorem~\ref{Theorem.ClopenDeterminacyGivesTruthPredicate}. Any definable winning strategy in the counting-down truth-telling game would provide a definable truth predicate, but by Tarski's theorem on the non-definability of truth, there is no such definable truth predicate. Thus, the counting-down truth-telling game is a first-order parameter-free definable clopen game in \ZFC, which can have no definable winning strategy (allowing parameters) for either player.
\end{proof}

It is interesting to observe that one may easily modify the truth-telling games by allowing a fixed class parameter $B$, so that clopen determinacy implies over \GBC\ that there is a satisfaction class relative to truth in $\langle V,{\in},B\rangle$. For example, we may get a truth predicate $\Tr_1$ for the structure $\<V,{\in},\Tr>$ itself, so that $\Tr_1$ concerns truth-about-truth. One may iterate this idea much further, to have predicates $\Tr_\alpha$ for every ordinal $\alpha$, which are truth predicates for the structure $\<V,{\in},\Tr_\beta>_{\beta<\alpha}$. Somewhat more uniformly, we may prefer a single binary predicate $\Tr\of\Ord\times V$, whose every slice $\Tr_\alpha=\set{x\mid \<\alpha,x>\in\Tr}$ is a truth predicate for the structure $\<V,{\in},\Tr\restrict(\alpha\times V)>$, and this is a more expressive treatment than having separate predicates, since one may now quantify over the earlier stages of truth. Indeed, one may hope to iterate truth predicates beyond $\Ord$ along any class well-order as in theorem~\ref{Theorem.ETR-iff-iterated-truth-predicate}.

Using the same ideas as in the proof of theorem~\ref{Theorem.ClopenDeterminacyGivesTruthPredicate}, one may formulate an \emph{iterated-truth-telling game}, where the truth-teller answers inquires about such iterated truth predicates, and then prove from clopen determinacy that there is indeed such an iterated truth predicate. This conclusion also follows immediately from theorem~\ref{Theorem.ClopenDetIffETR}, however, as the iterated truth predicate can be defined by an elementary transfinite recursion, and furthermore, the iterated-truth-telling game is fundamentally similar to the iteration game we use to prove theorem~\ref{Theorem.ClopenDetIffETR}. We shall therefore not give a separate proof for the iterated-truth case. Theorems~\ref{Theorem.ETR-iff-iterated-truth-predicate} and~\ref{Theorem.ClopenDetIffETR} show that clopen determinacy is equivalent over \GBC\ to the existence of iterated truth predicates over any class well-order.

Next, we briefly clarify the role of the global choice principle with the following well-known result.

\begin{theorem}[Folklore]\label{Theorem.ClopenDeterminacyImpliesAC}
 In \Godel-Bernays set theory \GB, the principle of clopen determinacy implies the global axiom of choice.
\end{theorem}

\begin{proof}
Consider the game where player I plays a nonempty set $b$ an player II plays a set $a$, with player II winning if $a\in b$. This is a clopen game, since it is over after one move for each player. Clearly, player I can have no winning strategy, since if $b$ is nonempty, then player II can win by playing any element $a\in b$. But a winning strategy for player II amounts exactly to a global choice function, selecting uniformly from each nonempty set an element.
\end{proof}

The set analogue of the proof of theorem~\ref{Theorem.ClopenDeterminacyImpliesAC} shows in \ZF\ that clopen determinacy for set-sized games implies the axiom of choice, and so over \ZF\ the principle of clopen determinacy for set-sized games is equivalent to the axiom of choice. As a consequence, we may prove in \ZF\ that the {\df universal axiom of determinacy}, which asserts that every game on every set is determined, is simply false: either there is some clopen game that is not determined, or the axiom of choice holds and there is a game on the natural numbers that is not determined. Pushing this idea a bit further leads to the following:

\begin{theorem}\label{Theorem.DefinableNondeterminedGame}
 There is a $\Delta^0_2$-definable set-sized game in \ZF\ that is not determined.
\end{theorem}

\begin{proof}
Let us emphasize that in the statement of the theorem we are referring to $\Delta^0_2$ in the sense of the \Levy\ hierarchy of the first-order language of set theory (and not in the sense of the arithmetic or projective hierarchies of descriptive set theory).

Consider the game $G$ where player I begins by playing a nonempty set of reals $A\of\omega^\omega$, with player II next playing an element of it $a\in A$; after this, player I plays a non-determined set $B\of\omega^\omega$, and then play proceeds as in the game $G_B$ determined by $B$. The first player to violate those requirements loses, and otherwise the winner is determined by the resulting play in $G_B$.

The winning condition for this game is $\Delta^0_2$-definable in set theory, because it is sufficiently local. For example, the assertion that a set $B\of\omega^\omega$ is non-determined has complexity $\Delta^0_2$ in set theory: ($\Pi^0_2$) for every strategy for one of the players, there is a strategy for the other that defeats it; ($\Sigma^0_2$) there is a set $M$ such that $M=V_{\omega+3}$ and $M$ satisfies ``$B$ is not determined.''

Finally, we argue in \ZF\ that neither player has a winning strategy for this game. Player I cannot have a winning strategy in $G$, since this would require him to play a non-determined set $B$ and then win the play of $G_B$, contradicting that $B$ is not determined. If player II has a winning strategy in $G$, then by the argument of theorem~\ref{Theorem.ClopenDeterminacyImpliesAC}, we get the axiom of choice for sets of reals and hence a well-ordering of the reals. From this, we know that there are non-determined sets of reals $B$, which player I can play, and then the winning strategy for player II would provide a winning strategy for $G_B$, contradicting the assumption that $B$ was not determined.
\end{proof}

A second simple observation about theorem~\ref{Theorem.ClopenDeterminacyImpliesAC} is that this argument answers a special case of question~\ref{Question.CanWeProveOpenClassDeterminacy?}. Namely, since there are models of \ZFC\ where the global axiom of choice fails for definable classes, there must be some models of \ZFC\ having definable clopen games with no definable winning strategy. Our theorem~\ref{Theorem.DefinableNondeterminedGame}, in contrast, establishes the stronger result that \emph{every} model of \ZFC, including those with global choice, has a definable clopen game with no definable strategy, and furthermore, the definition of the game is uniform.

\section{Clopen determinacy is equivalent over \GBC\ to \ETR}

We shall now generalize the argument of theorem~\ref{Theorem.ClopenDeterminacyGivesTruthPredicate} to prove a stronger result, which we believe explains the phenomenon of theorem~\ref{Theorem.ClopenDeterminacyGivesTruthPredicate}. Specifically, in theorem~\ref{Theorem.ClopenDetIffETR} we shall prove that clopen determinacy is exactly equivalent over \GBC\ to the principle of elementary transfinite recursion \ETR\ over well-founded class relations. This explains the result of theorem~\ref{Theorem.ClopenDeterminacyGivesTruthPredicate} because, as we have mentioned, truth itself is defined by such a recursion, namely, the familiar Tarskian recursive definition of truth defined by recursion on formulas, and so \ETR\ implies the existence of a satisfaction class for first-order truth.  Following Kentaro Fujimoto~\cite[Definition~88]{Fujimoto2012:Classes-and-truths-in-set-theory}, we introduce the principle of elementary transfinite recursion.
\begin{definition}\rm
 The principle of {\df elementary transfinite recursion} over well-founded class relations, denoted \ETR, is the assertion that every first-order recursive definition along any well-founded binary class relation has a solution.
\end{definition}

Let us explain in more detail. A binary relation $\lhd$ on a class $I$ is well-founded, if every nonempty subclass $B\of I$ has a $\lhd$-minimal element. This is equivalent in \GBC\ to the assertion that there is no infinite $\lhd$-descending sequence, and indeed one can prove this equivalence in $\GB+\DC$, meaning the dependent choice principle for set relations: clearly, if there is an infinite $\lhd$-descending sequence, then the set of elements on that sequence is a set with no $\lhd$-minimal element; conversely, if there is a nonempty class $B\of I$ with no $\lhd$-minimal element, then by the reflection principle relativized to the class $B$, there is some $V_\theta$ for which $B\intersect V_\theta$ is nonempty and has no $\lhd$-minimal element; but using \DC\ for $\lhd\intersect V_\theta$ we may successively pick $x_{n+1}\lhd x_n$ from $B\intersect V_\theta$, leading to an infinite $\lhd$-descending sequence. We find it interesting to notice that in this class context, therefore, well-foundedness for class relations becomes a first-order concept, which is a departure from the analogous situation in second-order number theory, where of course well-foundedness is $\Pi^1_1$-complete and definitely not first-order expressible in number theory.

Continuing with our discussion of recursion, suppose that we have a well-founded binary relation $\lhd$ on a class $I$, and suppose further that $\varphi(x,b,F,Z)$ is a formula describing the recursion rule we intend to implement, where $\varphi$ involves only first-order quantifiers, $F$ is a class variable for a partial solution and $Z$ is a fixed class parameter, henceforth suppressed. The idea is that $\varphi(x,b,F)$ expresses the recursive rule to be iterated. Namely, a {\df solution} of the recursion is a class $F\of I\times V$ such that for every $b\in I$, the $b^{\rm th}$ slice of the solution $F_b=\set{x\mid \varphi(x,b,F\restrict b)}$ is defined by the recursive rule, where $F\restrict b=F\intersect(\set{c\in I\mid c\lhd b}\times V)$ is the partial solution on slices prior to $b$. In this way, each slice of the solution $F_b$ is determined via the recursive rule from the values on the slices $F_c$ for earlier values $c$.

The principle of elementary transfinite recursion \ETR\ over well-founded class relations asserts that for every such well-founded relation $\langle I,{\lhd}\rangle$ and any first-order recursive rule $\varphi$ as above, there is a solution. One may equivalently consider only well-founded partial-order relations, or well-founded tree orders, or class well-orders.

\begin{lemma}
 The principle of \ETR\ is equivalently formulated over \GBC\ with any of the following types of well-founded relations:
 \begin{enumerate}
  \item well-founded class relations.
  \item well-founded class partial orders.
  \item well-founded class tree orders.
  \item well-ordered class relations.
 \end{enumerate}
\end{lemma}

\begin{proof}
Since the families of relations are becoming more specialized, it is clear that each statement implies the next $(1\to 2\to 3\to 4)$. It remains to prove the converse implications.

$(2\to 1)$ If $\lhd$ is any well-founded binary class relation, then let $<$ be the transitive closure of $\lhd$, which is a well-founded partial order. Any first-order recursion on $\lhd$ is easily transformed to a recursion on $<$, using $\lhd$ as a class parameter if necessary.

$(3\to 2)$. If $<$ is any well-founded partial order, then let $T$ be the tree of all finite $<$-descending sequences, ordered by extension, so that longer is lower. This tree is well-founded, and any recursion defined on $<$ can be easily transferred to the tree order.

$(4\to 3)$ Suppose that we have a first-order recursion defined on a well-founded class tree $T$ of sequences (the root is at top, the tree grows downward). We can linearize the tree by the Kleene-Brouwer order, by which $s$ is less than $t$ if $s$ extends $t$ or if when they disagree, then $s$ is lower than $t$ in that coordinate, with respect to a fixed global well-ordering of the universe. This is a class well-order, and it is easy to transfer the recursion from $T$ to this linear order.
\end{proof}
Note that if $\GBC$ is consistent, then it does not prove that all recursions along set well-orders have a solution,  because a recursion of length $\omega$ suffices to define a truth predicate by the Tarskian recursion on formulas, and the existence of such a predicate implies $\Con(\ZFC)$ and therefore also $\Con(\GBC)$. So \GBC\ plus \ETR\ is strictly stronger than \GBC\ in consistency strength, although it is provable in Kelley-Morse set theory \KM, in essentially the same way that \GBC\ proves the set-like special case.

Next, let's explain the tight connection between \ETR\ and the existence of iterated truth predicates, a result similar to those obtained by Fujimoto~\cite{Fujimoto2012:Classes-and-truths-in-set-theory} (for instance, see his corollary 61). If $\<I,\lhd>$ is a class well-order, then $T$ is an {\df iterated truth predicate} over $\<I,\lhd>$ relative to $Z$, if for each $i\in I$, the $i^{\rm th}$ slice $T_i$ is a truth predicate for the structure $\<V,{\in},Z,T\restrict i>$, satisfying the Tarskian recursion for formulas in the language of set theory augmented with predicates for $Z$ and $T\restrict i$, where $T\restrict i=T\intersect(\set{j\in I\mid j\lhd i}\times V)$ is the restriction of $T$ to the $\lhd$-earlier stages of truth.

\begin{theorem}\label{Theorem.ETR-iff-iterated-truth-predicate}
The principle \ETR\ of elementary transfinite recursion is equivalent over \GBC\ to the assertion that for every class parameter $Z$ and every class well-ordering $\<I,\lhd>$ there is an iterated truth predicate $T$ along $\<I,\lhd>$ over the parameter $Z$.
\end{theorem}

\begin{proof}
The forward implication is straightforward, since the truth predicate itself is defined by a transfinite recursion of length $\omega\cdot \<I,\lhd>$. Namely, to get the truth predicate for the next stage, one simply performs the Tarskian recursion through the formula complexity hierarchy, which has height $\omega$.

Conversely, suppose that for class parameter $Z$ and class well-ordering $\<I,\lhd>$, we have an iterated truth predicate $T$ over $\lhd$ relative to $Z$. Now suppose that we have an instance of \ETR\, iterating a formula $\varphi(x,i,F,Z)$ along $\<I,\lhd>$. We claim that from parameter $T$, we may define a solution $F$ to this recursion. Specifically, we claim that there is a formula $\bar\varphi$ such that if one extracts from $T$ the class defined by $\bar\varphi$, namely, $F=\set{\<i,x>\mid T(i,\<\bar\varphi,x>)}$, then $F$ is a solution to the recursion of $\varphi$ along $\lhd$. The formula $\bar\varphi$ should simply be chosen so that $\<V,{\in},Z,T\restrict i>\satisfies\bar\varphi(x,i)$ if and only if $\<V,{\in},Z,F\restrict i>\satisfies\varphi(x,i)$, where $F$ is defined as just mentioned using $\bar\varphi$. Such a formula $\bar\varphi$ exists by \Godel's fixed-point lemma: for any $e$, let $\psi(e,x,i)$ be the assertion $\<V,{\in},Z,\set{\<i,x>\mid T(i,\<e,x>)}>\satisfies\varphi(x,i)$, and then by the usual fixed-point trick find a formula $\bar\varphi(x,i)$, for which $\<V,{\in},Z,T\restrict i>\satisfies\psi(\bar\varphi,x,i)\iff\bar\varphi(x,i)$. It follows that the class $F$ iteratively defined from $T$ by $\bar\varphi$ satisfies $\varphi$ at each step and therefore is a solution to the recursion of $\varphi$ along $\lhd$, as desired.
\end{proof}

An important immediate consequence of theorem~\ref{Theorem.ETR-iff-iterated-truth-predicate} is that \ETR\ is expressible as a single second-order assertion in the language of \GBC, and so we needn't treat it as a scheme (see also the remarks at the beginning of the proof of~\cite[theorem~88]{Fujimoto2012:Classes-and-truths-in-set-theory}). Namely, \ETR\ is equivalent to the assertion that for every class well-order $\<I,\lhd>$ and every class parameter $Z$, there is an iterated truth-predicate $T$ along $\<I,\lhd>$ relative to $Z$.

We come now to the next main contribution of this article, the equivalence of clopen determinacy for class games with \ETR, and with the existence of iterated truth predicates over class well-orders.

\goodbreak
\begin{theorem}\label{Theorem.ClopenDetIffETR}
In \Godel-Bernays set theory \GBC, the following are equivalent.
\begin{enumerate}
 \item Clopen determinacy for class games. That is, in any two-player game of perfect information whose winning condition class is both open and closed, there is a winning strategy for one of the players.
 \item The principle \ETR\ of elementary transfinite recursion over well-founded class relations: every such recursion has a solution.
 \item Existence of iterated-truth predicates. That is, for every class parameter $Z$ and every class well-ordering $\<I,\lhd>$, there is an iterated truth predicate $T$ along $\<I,\lhd>$ over parameter $Z$.
\end{enumerate}
\end{theorem}

\begin{proof}
($2\leftrightarrow 3$) This is established by theorem~\ref{Theorem.ETR-iff-iterated-truth-predicate}.

($2\to 1$) Assume the principle \ETR\ of elementary transfinite recursion, and suppose we are faced with a clopen game. Consider the game tree, consisting of positions arising during play, up to the moment that a winner is known, orienting the tree so that the root is at the top and play proceeds downward. This tree is well-founded precisely because the game is clopen. Let us label the terminal nodes of the tree with I or II according to who has won the game in that position, and more generally, let us label all the nodes of the tree with I or II according to the following transfinite recursion: if a node has I to play, then it will have label I if there is a move to a node already labeled I, and otherwise II; similarly, when it is player II's turn to play, then if she can play to a node labeled II, we label the original node with II, and otherwise I. By the principle of elementary transfinite recursion, there is a labeling of the entire tree that accords with this recursive rule. It is now easy to see that if the initial node is labeled with I, then player I has a winning strategy, which is simply to stay on the nodes labeled I. (We use the global choice principle to choose a particular such node with the right label; this use can be avoided if the space $X$ of possible moves is already well-ordered, such as in the case of games on the ordinals $X=\Ord$.) Note that player II cannot play in one move on her turn from a node labeled I to one labeled II. Similarly, if the initial node is labeled II, then player II has a winning strategy, which is simply to stay on the nodes labeled II. And so the game is determined, and we have established clopen determinacy.

($1\to 2$) This implication is the main new content of this theorem. Assume the principle of clopen determinacy for class games, and suppose that we are faced with a recursion along a well-founded class partial-order relation $\lhd$ on a class $I$, using a first-order recursion rule $\varphi(x,b,F)$, possibly with a fixed class parameter $Z$, which we suppress. We shall define a certain clopen game, and prove that any winning strategy for this game will produce a solution for the recursion.

At first, we consider a simpler open game, the {\df recursion game}, which will be much like the truth-telling game used in theorem~\ref{Theorem.ClopenDeterminacyGivesTruthPredicate}, except that in this game, the truth-teller will also provide information about the putative solution of the recursion in question; later, we shall revise this game to a clopen game. In the recursion game, we have the same two players again, the interrogator and the truth-teller, but now the interrogator will make inquiries about truth in a structure of the form $\langle V,{\in},{\lhd},F\rangle$, where $\lhd$ is the well-founded class relation and $F$ is a binary class predicate, not yet specified, but which we hope will become a solution of the recursion, with $F\of I\times V$. Specifically, the interrogator is allowed to ask about the truth of any first-order formula $\varphi(\vec a)$ in the language of this structure and in particular to inquire as to whether $F(i,x)$ or not. The truth-teller, as before, will answer the inquiries by pronouncing either that $\varphi(\vec a)$ is true or that it is false, and in the case $\varphi(\vec a)=\exists x\,\psi(x,\vec a)$ and the formula was pronounced true, then the truth-teller shall also provide as before a witness $b$ for which she also pronounces $\psi(b,\vec a)$ to be true. The truth-teller loses immediately, if she should ever violate Tarski's recursive definition of truth, and she also is required to pronounce any instance of the recursion rule $F(i,x)\leftrightarrow\varphi(x,i,F\restrict i)$ to be true, where $F\restrict i$ denotes the class $F\intersect(\set{j\in I\mid j\lhd i}\times V)$. Specifically, we form the formula $\varphi(x,i,{F\restrict i})$ in the language of set theory with a predicate for $F$ by replacing any atomic occurrence of the predicate $F(j,y)$ in $\varphi$ with $F(j,y)\wedge j\lhd i$. Since violations of any of these requirements, if they occur at all, do so at a finite stage of play, it follows that the game is open for the interrogator.

\begin{sublemma}
 The interrogator has no winning strategy in the recursion game.
\end{sublemma}

\begin{proof}
To prove this lemma, we use a modification of the idea of lemma~\ref{Lemma.InterrogatorHasNoWinningStrategy}. Suppose that $\sigma$ is a strategy for the interrogator. So $\sigma$ is a class function that instructs the interrogator how to play next, given a position of partial play. By the reflection theorem, there is an ordinal $\theta$ such that $V_\theta$ is closed under $\sigma$, and using the satisfaction class that comes from clopen determinacy, we may actually also arrange that $\langle V_\theta,{\in},{\lhd}\cap V_\theta,\sigma\cap V_\theta\rangle\prec\langle V,{\in},{\lhd},\sigma\rangle$. Consider the relation $\lhd\cap V_\theta$, which is a well-founded relation on $I\cap V_\theta$. Since $\ZFC$ and hence $\GBC$ proves the existence of solutions to transfinite recursions for sets, there is a (unique) solution $f\of (I\cap V_\theta)\times V_\theta$ such that the $i^{\rm th}$ slice $f_i=\set{x\in V_\theta\mid \langle V_\theta,{\in},{\lhd}\cap V_\theta,f\rangle\satisfies\varphi(x,i,f\restrict i)}$ is defined by the recursion of $\varphi$, where $f\restrict i$ means the restriction of the predicate $f\intersect(\set{j\in I\intersect V_\theta\mid j\lhd i}\times V_\theta)$ to the predecessors of $i$ that are also in $V_\theta$. Consider now the play of the recursion game in $V$, where the interrogator uses the strategy $\sigma$ and the truth-teller plays in accordance with truth in the structure $\langle V_\theta,{\in},{\lhd}\cap V_\theta,f\rangle$, which is a little sneaky because the function $f$ is a solution of the recursion rule $\varphi$ only on the relation $\lhd\intersect V_\theta$, rather than the full relation $\lhd$. But since $V_\theta$ was closed under $\sigma$, the interrogator will never issue challenges outside of $V_\theta$ in this play; and since the function $f$ fulfills the recursive rule $f(i,x)\leftrightarrow\varphi(x,i,f\restrict i)$ in this structure, the truth-teller will not be trapped in any violation of the Tarski conditions or the recursion condition. Thus, the truth-teller will win this instance of the game, and so $\sigma$ is not a winning strategy for the interrogator, as desired.
\end{proof}

\begin{sublemma}
 The truth-teller has a winning strategy in the recursion game if and only if there is a solution of the recursion.
\end{sublemma}

\begin{proof}
If there is a solution $F$ of the recursion, then by clopen determinacy we know there is also a satisfaction class $\Tr$ for first order truth in the structure $\langle V,{\in},{\lhd},F\rangle$, and the truth-teller can answer all queries of the interrogator in the recursion game by referring to what $\Tr$ asserts is true in this structure. This will be winning for the truth-teller, since $\Tr$ obeys the Tarskian conditions and makes all instances of the recursive rule true with the predicate $F$.

Conversely, suppose that $\tau$ is a winning strategy for the truth-teller in the recursion game. We may see as before that the truth pronouncements made by $\tau$ about truth in the structure $\<V,{\in},\lhd>$ are independent of the play in which they occur, and they provide a satisfaction class for this structure. This is proved just as for the truth-telling game by induction on the complexity of the formulas: the strategy must correctly answer all atomic formulas, and the answers to more complex formulas must be independent of the play since violations of this would lead to violations of the Tarski conditions by reducing to simpler formulas, as before, and this would contradict our assumption that $\tau$ is a winning strategy for the truth-teller.

Consider next the truth pronouncements made by $\tau$ in the language involving the class predicate symbol $F$. We shall actually need this property only in the restricted languages, where for each $i\in I$, we consider formulas asserting truth in the structure $\<V,{\in},{\lhd},F\restrict i>$, rather than concerning truth in the full structure $\<V,{\in},{\lhd},F>$. We claim by induction on $i$, with an embedded induction on formulas, that for every $i\in I$, the truth pronouncements provided by the strategy $\tau$ in this language are independent of the play in which they are made and furthermore provide a truth predicate for a structure of the form $\<V,{\in},\lhd,F\restrict i>$. The case where $i$ is $\lhd$-minimal is essentially similar to the case we already handled, where no reference to $F$ is made, since $F\restrict i$ must be asserted to be empty in this case. Suppose inductively that our claim is true for assertions in the language with $F\restrict j$, whenever $j\lhd i$, and consider the language with $F\restrict i$. (Note that the claim we are proving by induction is first-order expressible in the class parameter $\tau$, and so this induction can be legitimately undertaken in \GBC; we haven't allowed an instance of $\Pi^1_1$-comprehension to sneak in here.) It is not difficult to see that $\tau$ must pronounce that the various predicates $F\restrict j$ cohere with one another on their common domain, since any violation of this will give rise to a violation of the Tarskian recursion. So our induction assumption ensures that $\tau$ has determined a well-defined class predicate $F\restrict i$. Furthermore, since $\tau$ is required to affirm that $F$ obeys the recursive rule, it follows that $\tau$ asserts that $F\restrict i$ obeys the recursive rule up to $i$.

We now argue by induction on formulas that the truth pronouncements made by $\tau$ about the structure $\<V,{\in},\lhd,F\restrict i>$ forms a satisfaction class for this structure. In the atomic case, the truth pronouncements about this structure are independent of the play of the game in which they occur, since this is true for atomic formulas in the language of set theory and for atomic assertions about $\lhd$, by the rules of the game, and it true for atomic assertions about $F\restrict i$ by our induction hypothesis on $i$. Continuing the induction, it follows that the truth pronouncements made about compound formulas in this structure are similarly independent of the play and obey the Tarskian conditions, since any violation of this can be easily exposed by having the interrogator inquire about the constituent formulas, just as in the truth-telling game. So the claim is also true for $F\restrict i$.

Thus, for every $i\in I$, the strategy $\tau$ is providing a satisfaction class for the structure $\<V,{\in},\lhd,F\restrict i>$, which furthermore verifies that the resulting class predicate $F\restrict i$ determined by this satisfaction class fulfills the desired recursion relation up to $i$. Since these restrictions $F\restrict i$ also all agree with one another, the union of these class predicates is a class predicate $F\of I\times V$ that for every $i$ obeys the desired recursive rule $F_i=\set{x\mid \varphi(x,i,F\restrict i)}$. So the recursion has a solution, and this instance of the principle of first-order transfinite recursion along well-founded class relations is true.
\end{proof}

So far, we have established that the principle of open determinacy implies the principle \ETR\ of elementary transfinite recursion. In order to improve this implication to use only clopen determinacy rather than open determinacy, we modify the game as in lemma~\ref{Lemma.TruthTellerStrategyIffSatisfactionClass} by requiring the interrogator to count-down during play. Specifically, the {\df count-down recursion game} proceeds just like the recursion game, except that now we also insist that the interrogator announce on the first move a natural number $n$, such that the interrogator loses if the truth-teller survives for at least $n$ moves (we could have had him count down in the ordinals instead, which would have made things more flexible for him, but the analysis is essentially the same). This is now a clopen game, since the winner will be known by the time this clock expires, either because the truth-teller will violate the Tarski conditions or the recursion condition before that time, in which case the interrogator wins, or else because she did not and the clock expired, in which case the truth-teller wins. So this is a clopen game.

Since the modified version of the game is even harder for the interrogator, there can still be no winning strategy for the interrogator. So by the principle of clopen determinacy, there is a winning strategy $\tau$ for the truth-teller. This strategy is allowed to make decisions based on the number $n$ announced by the interrogator on the first move, and it will no longer necessarily be the case that the theory declared true by $\tau$ will be independent of the interrogator's play, since the truth-teller can relax as the time is about to expire, knowing that there isn't time to be caught in a violation. Nevertheless, it will be the case, we claim, that the theory pronounced true by $\tau$ for all plays with sufficiently many remaining moves will be independent of the interrogator's play. One can see this by observing that if an assertion $\psi(\vec a)$ is independent in this sense, then also $\neg\psi(\vec a)$ will be independent in this sense, for otherwise there would be plays with a large number of plays remaining giving different answers for $\neg\psi(\vec a)$ and we could then challenge directly afterward with $\psi(\vec a)$, which would have to give different answers or else $\tau$ would not win. Similarly, since $\tau$ is winning for the truth-teller, one can see that allowing the interrogator to specify a bound on the total length of play does not prevent the arguments above showing that $\tau$ describes a coherent solution predicate $F\of I\times V$ satisfying the recursion $F(i,x)\leftrightarrow\varphi(x,i,F\restrict i)$, provided that one looks only at plays in which there are sufficiently many moves remaining. There cannot be a $\lhd$-least $i$ where the value of $F(i,x)$ is not determined in this sense, and so on just as before. So the strategy must give us a class predicate $F$ and a truth predicate for $\<V,{\in},{\lhd},F>$ witnessing that it solves the desired recursion, as desired.

In conclusion, the principle of clopen determinacy for class games is equivalent to the principle \ETR\ of elementary transfinite recursion along well-founded class relations.
\end{proof}

\section{Proving open determinacy in strong theories}

It follows from theorems~\ref{Theorem.ClopenDeterminacyGivesTruthPredicate} and~\ref{Theorem.ClopenDetIffETR} that the principle of open determinacy for class games cannot be proved in set theories such as \ZFC\ or \GBC, if these theories are consistent, since there are models of those theories that have no satisfaction class for first-order truth. We should now like to prove, in contrast, that the principle of open determinacy for class games \emph{can} be proved in stronger set theories, such as Kelley-Morse set theory \KM, as well as in $\GBC+\Pi^1_1$-comprehension, which is a proper fragment of \KM.

In order to undertake this argument, however, it will be convenient to consider the theory $\KM^+$, a natural strengthening of Kelley-Morse set theory \KM\ that we consider in~\cite{GitmanHamkinsJohnstone:Kelley-MorseSetTheoryAndChoicePrinciplesForClasses}. The theory $\KM^+$ extends \KM\ by adding the {\df class-choice scheme}, which asserts of any second-order formula $\varphi$, that for every class parameter $Z$, if for every set $x$ there is a class $X$ with property $\varphi(x,X,Z)$, then there is a class $Y\of V\times V$, such that for every $x$ we have $\varphi(x,Y_x,Z)$, where $Y_x$ denotes the $x^{\rm th}$ slice of $Y$. Thus, the axiom asserts that if every set $x$ has a class $X$ with a certain property, then we can choose particular such classes and put them together into a single class $Y$ in the plane, such that the $x^{\rm th}$ slice $Y_x$ is a witness for $x$. In~\cite{GitmanHamkinsJohnstone:Kelley-MorseSetTheoryAndChoicePrinciplesForClasses}, we prove that this axiom is not provable in \KM\ itself, thereby revealing what may be considered an unfortunate weakness of \KM. The class-choice scheme can also naturally be viewed as a class collection axiom, for the class $Y$ gathers together a sufficient collection of classes $Y_x$ witnessing the properties $\varphi(x,Y_x,Z)$. In this light, the weakness of \KM\ in comparison with $\KM^+$ is precisely analogous to the weakness of the theory $\ZFCmm$ in comparison with the theory $\ZFCm$ that we identified in~\cite{GitmanHamkinsJohnstone:WhatIsTheTheoryZFC-Powerset?}---these are the theories of \ZFC\ without power set, using replacement or collection $+$ separation, respectively---since in each case the flawed weaker theory has replacement but not collection, which leads to various unexpected failures for the respective former theories.

The natural weakening of the class-choice scheme to the case where $\varphi$ is a first-order assertion, having only set quantifiers, is called the {\df first-order class-choice principle}, and it is expressible as a single assertion, rather than only as a scheme, in $\KM$ and indeed in $\GBC+\ETR$, since in these theories we have first-order truth-predicates available relative to any class. A still weaker axiom makes the assertion only for choices over a fixed set, such as the first-order class $\omega$-choice principle: $$\forall Z\left(\strut\forall n\in\omega\,\exists X\,\varphi(n,X,Z)\implies\exists Y\of\omega\times V\, \forall n\in\omega\,\varphi(n,Y_n,Z)\right),$$where $\varphi$ has only first-order quantifiers, and this is also finitely expressible in $\GBC+\ETR$. In our paper~\cite{GitmanHamkinsJohnstone:Kelley-MorseSetTheoryAndChoicePrinciplesForClasses}, we separate these axioms from one another and prove that none of them is provable in \KM, assuming the consistency of an inaccessible cardinal.

The {\df $\Pi^1_1$-comprehension axiom} is the assertion that for any $\Pi^1_1$ formula $\varphi(x,Z)$, with class parameter $Z$, we may form the class $\set{a\mid \varphi(a,Z)}$. By taking complements, this is equivalent to $\Sigma^1_1$-comprehension.

\begin{theorem}\label{Theorem.KMImpliesOpenDeterminacy}
 Kelley-Morse set theory $\KM$ proves the principle of open determinacy for class games. Indeed, this conclusion is provable in the subtheory consisting of $\GBC$ plus $\Pi^1_1$-comprehension.
\end{theorem}

\noindent We are unsure whether the strictly weaker theory of $\GBC+\ETR$ suffices for this conclusion; see additional discussion in section \ref{Section.Questions}.

\begin{proof}
Assume \GBC\ plus $\Pi^1_1$-comprehension. In order to make our main argument more transparent, we shall at first undertake it with the additional assumption that the first-order class-choice principle holds (thus, we work initially in a fragment of $\KM^+$). Afterwards, we shall explain how to eliminate our need for the class-choice principle, and thereby arrive at a proof using just $\GBC$ plus $\Pi^1_1$-comprehension.

Consider any open class game $A\of X^\omega$, where $A$ is the open winning condition and $X$ is the class of allowed moves. We shall show the game is determined. To do so, notice that for any position $p\in X^\ltomega$, the assertion that a particular class function $\sigma:X^\ltomega\to X$ is a winning strategy for player I in the game proceeding from position $p$ is an assertion about $\sigma$ involving only first-order quantifiers; one must say simply that every play of the game that proceeds from $p$ and follows $\sigma$ on player I's moves after that, is in $A$. Thus, the assertion that player I has a winning strategy for the game starting from position $p$ is a $\Sigma^1_1$ assertion about $p$. Using $\Pi^1_1$-comprehension, therefore, we may form the class
 $$W=\left\{\ p\in X^\ltomega\ \mathrel{\hbox{\huge$|$}}\quad\raise.5em\vtop{\hbox{Player I has a winning strategy in the}\hbox{game proceeding from position $p$}}\ \right\}.$$
With this class, we may now carry out a class analogue of one of the usual soft proofs of the Gale-Stewart theorem, which we mentioned in the introduction of this article. Namely, if the initial (empty) position of the game is in $W$, then player I has a winning strategy, and we are done. Otherwise, the initial node is not in $W$, and we simply direct player II to avoid the nodes of $W$ during play. If this is possible, then it is clearly winning for player II, since he will never land on a node all of whose extensions are in the open class, since such a node is definitely in $W$, and so he will win. To see that player II can avoid the nodes of $W$, observe simply that at any position $p$, if it is player II's turn to play, and player I does not have a strategy in the game proceeding from $p$, then we claim that there must be at least one move that player II can make, to position $p\concat x$ for some $x\in X$, such that $p\concat x\notin W$. If not, then $p\concat x\in W$ for all moves $x$, and so for each such $x$ there is a strategy $\tau_x$ that is winning for player I in the game proceeding from $p\concat x$. By the first-order class-choice principle (and this is precisely where we use our extra assumption), we may gather such strategies $\tau_x$ together into a single class and thereby construct a strategy for player I that proceeds from position $p$ in such a way that if player II plays $x$, then player I follows $\tau_x$, which is winning for player I. Thus, there is a winning strategy for player I from position $p$, contradicting our assumption that $p\notin W$, and thereby establishing our claim. So if $p\notin W$ and it is player II's turn to play, then there is a play $p\concat x$ that remains outside of $W$. Similarly, if $p\notin W$ and it is player I's turn to play, then clearly there can be no next move $p\concat y$ placing it inside $W$, for then player I would have also had a strategy from position $p$. Thus, if the initial position is not in $W$, then player II can play so as to retain that property (using global choice to pick a particular move realizing that situation), and player I cannot play so as to get inside $W$, and this is therefore a winning strategy for player II. So the game is determined.

The argument above took place in the theory $\GBC+\Pi^1_1$-comprehension $+$ the first-order class-choice principle. And although it may appear at first to have made a fundamental use of the class-choice principle, we shall nevertheless explain how to eliminate this use. The first observation to make is that $\Pi^1_1$-comprehension implies the principle \ETR\  of elementary transfinite recursion along any well-founded relation. To see this, suppose that $\lhd$ is any well-founded relation on a class $I$ and $\varphi(x,i,F,Z)$ is a formula to serve as the recursive rule. The class of $i\in I$ that are in the domain of some partial solution to the recursion is $\Sigma^1_1$-definable. And furthermore, all such partial solutions must agree on their common domain, by an easy inductive argument along $\lhd$. It follows that the union of all the partial solutions is $\Sigma^1_1$-definable and therefore exists as a class, and it is easily seen to obey the recursion rule on its domain. So it is a maximal partial solution. If it is not defined on every slice in $I$, then there must be a $\lhd$-minimal element $i$ whose $i^{\rm th}$ slice is not defined; but this is impossible, since we could apply the recursive rule once more to determine the slice $F_i=\set{x\mid \varphi(x,i,F\restrict i,Z)}$ to place at $i$, thereby producing a partial solution that includes $i$. So the maximal partial solution is actually a total solution, verifying this instance of the transfinite recursion principle.

Next, we shall explain how to continue the constructibility hierarchy beyond $\Ord$. This construction has evidently been discovered and rediscovered several times in set theory, but rarely published; the earliest reference appears to be the dissertation of Leslie Tharp~\cite{Tharp1965:ConstructibilityInImpredicativeSetTheory}, although Bob Solovay reportedly also undertook the construction as an undergraduate student, without publishing it. In the countable realm, of course, the analogous construction is routine, where one uses reals to code arbitrary countable structures including models of set theory of the form $\<L_\alpha,\in>$. For classes, suppose that $\Gamma=\<\Ord,\leq_\Gamma>$ is a {\df meta-ordinal}, which is to say, a well-ordered class relation $\leq_\Gamma$ on $\Ord$; this relation need not necessarily be set-like, and the order type can reach beyond $\Ord$. By \ETR, we may iterate the constructible hierarchy up to $\Gamma$, and thereby produce a class model $\<L_\Gamma,\in_\Gamma>$ of $V=L$, whose ordinals have order-type $\Gamma$. Specifically, we reserve a class of nodes to be used for representing the new ``(meta-)sets'' at each level of the $L_\Gamma$ hierarchy, and define $\in_\Gamma$ recursively, so that at each level, we add all and only the sets that are definable (from parameters) over the previous structure. To be clear, the domain of the structure $\<L_\Gamma,{\in_\Gamma}>$ is a class $L_\Gamma\of V$, and the relation $\in_\Gamma$ is not the actual $\in$ relation, but nevertheless $\in_\Gamma$ is a well-founded extensional relation in our original model, and the structure $\<L_\Gamma,\in_\Gamma>$ looks internally like the constructible universe. Thus, we have what might be termed merely a code for or presentation of the fragment $L_\Gamma$ of the constructibility hierarchy up to $\Gamma$, which someone outside the universe might prefer to think of as an actual transitive set.

In order to speak of $L_\Gamma$ in our \GBC\ context, then, we must be aware that different choices of $\Gamma$ will lead to different presentations, with sets being represented differently and by different sets. Nevertheless, we may assume without loss that the actual sets in $L=L_{\Ord}$ are represented in this presentation in some highly canonical way, so that the ordinals are represented by themselves, for example, and the other sets are represented by their own singletons (say), and so in particular, all the various $L_\Gamma$ will agree on their $\in_\Gamma$ relations for sets constructed before $\Ord$. Also, using the principle of first-order transfinite recursion, it is easy to see that any two meta-ordinals $\Gamma$ and $\Gamma'$ are comparable, in the sense that one of them is (uniquely) isomorphic to an initial segment of the other, and similarly the structures $L_\Gamma$ and $L_{\Gamma'}$ admit such coherence as well; in particular, if $\Gamma$ and $\Gamma'$ are isomorphic, then so also are the structures $L_\Gamma$ and $L_{\Gamma'}$. Consider the meta-ordinals $\Gamma$ for which there is a larger meta-ordinal $\Theta$, such that $L_\Theta$ has as an element a well-ordered structure $\Gamma'=\<\Ord,\leq_{\Gamma'}>$ with order-type isomorphic to $\Gamma$. In a sense, these are the meta-ordinals below $(\Ord^+)^L$. In this case, there will be an $L$-least such code $\Gamma'$ in $L_\Theta$. And furthermore, any other meta-ordinal $\Theta'$ which constructs such a code will agree on this $L$-least code. Since we assumed that the ordinals were represented as themselves in $L_\Theta$, we may view $\Gamma'$ as a meta-ordinal in our original model. Thus, the meta-ordinals $\Gamma$ realized by a relation in some $L_\Theta$ have \emph{canonical} codes, the meta-ordinals that are $L$-least with respect to some (and hence all sufficiently large) $L_\Theta$. There is exactly one such code for each meta-ordinal order type that is realized inside any $L_\Theta$.

Now, let $\mathcal L$ be the collection of classes $B\of L$ that are realized as an element in some $L_\Gamma$---these are the classes of the meta-$L$ that are contained in the actual $L$---and consider the model $\mathscr L=\<L,{\in},\mathcal L>$. It is not difficult to see that this is a model of \GBC, precisely because the $L$-hierarchy closes under definability at each step of the recursion. Furthermore, the existence of canonical codes will allow us to show that this model satisfies the first-order class-choice principle. Suppose that $\mathscr L\satisfies \forall b\,\exists X\, \varphi(b,X)$, where $\varphi$ has only first-order quantifiers. For each set $b$, there is a class $X$ with property $\varphi(b,X)$, and such a class $X$ exists as a set in some $L_\Gamma$ for some meta-ordinal $\Gamma$. We may consider $\Gamma$ to be a canonical code for a meta-ordinal which is minimal with respect to the property of having such an $X$, and in this case, the class $\Gamma=\Gamma_b$ is $\Sigma^1_1$-definable (and actually $\Delta^1_1$-definable) from $b$. So the map $b\mapsto\Gamma_b$ exists as a class in the ground model, and we may therefore form a meta-ordinal $\Theta$ that is larger than all the resulting meta-ordinals $\Gamma_b$. Inside $L_\Theta$, we may select the $L$-least $X_b$ witnessing $\varphi(b,X_b,Z)$, and thereby form the class $\set{(b,c)\mid c\in X_b}$, which fulfills this instance of the first-order class-choice principle (and the argument easily accommodates class parameters).

A similar idea shows that $\mathscr L$ satisfies $\Pi^1_1$-comprehension, provided that this was true in the original model (and indeed $\mathscr L$ satisfies \KM, if this was true in the original model, and in this case one can also verify the class-choice scheme in $\mathscr L$, without requiring this in $V$, which shows that $\Con(\KM)\implies\Con(\KM^+)$; see~\cite{GitmanHamkinsJohnstone:Kelley-MorseSetTheoryAndChoicePrinciplesForClasses}.) It will be more convenient to establish $\Sigma^1_1$-comprehension, which is equivalent. Consider a formula of the form $\exists X\, \varphi(b,X)$, where $\varphi$ has only first-order quantifiers. The class $B=\set{b\in L\mid \exists X\in\mathcal{L}\,\,\varphi(b,X)}$ is $\Sigma^1_1$-definable and therefore exists as a class in our original universe. We need to show it is in $\mathcal{L}$. For each $b\in B$, there is a class $X$ such that $X\in\mathcal{L}$ and $\varphi(b,X)$, and such a class $X$ is constructed in some $L_\Gamma$ at some minimal meta-ordinal stage $\Gamma$, which we may assume is a canonical code. Thus, the map $b\mapsto \Gamma_b$ is $\Sigma^1_1$-definable, and so it exists as a class. Thus, we may form a single meta-ordinal $\Theta$ larger than all the $\Gamma_b$, and in $L_\Theta$, we may define the set $B$. So $B\in\mathcal{L}$, verifying this instance of $\Sigma^1_1$-comprehension in $\mathscr{L}$, as desired.

One may now check that the construction of the previous paragraph relativizes to any class $Z\of\Ord$, leading to a model $\<L[Z],{\in},\mathcal{S}>$ that satisfies $\GBC+\Pi^1_1$-comprehension $+$ the first-order class-choice principle, in which $Z$ is a class. One simply carries the class parameter $Z$ through all of the previous arguments. If $Z$ codes all of $V$, then the result is a model $\<V,{\in},\mathcal{S}>$, whose first-order part has the same sets as the original model $V$.

Using this, we may now prove the theorem. Consider any open class game $A\of X^\omega$, where $X$ is the class of allowed moves. Let $Z\of\Ord$ be a class that codes in some canonical way every set in $V$ and also the classes $X$ and $A$. The resulting structure $\<V,{\in},\mathcal{S}>$ described in the previous paragraph therefore satisfies $\GBC+\Pi^1_1$-comprehension $+$ the first-order class-choice principle. The game $A\of X^\omega$ exists in this structure, and since it is open there, it follows by the first part of the proof of this theorem that this game is determined in that structure. So there is a strategy $\sigma\in\mathcal{S}$ for the game $A$ that is winning for one of the players. But this is absolute to our original universe, because the two universes have exactly the same sets and therefore exactly the same plays of the game. So the game is also determined in our original universe, and we have thus verified this instance of the principle of open determinacy for class games.
\end{proof}

One might view the previous argument as a proper class analogue of Blass's result~\cite{Blass1972:ComplexityOfWinningStrategies} that computable games have their winning strategies appearing in the $L$-hierarchy before the next admissible set, since we found the winning strategies for the open class game in the meta-$L$ hierarchy on top of the universe. Nevertheless, we are unsure sure exactly what it takes in the background theory to ensure that the meta-$L$ structure is actually admissible.

\section{The strength of ETR}

To provide further context of the second-order set theories mentioned in this article, we should like next to explain Kentaro Sato's surprising result separating $\ETR$ from $\Delta^1_1$-comprehension over \GBC. This result is surprising, because it is an instance where the analogy between results in second-order arithmetic and those in second-order set theory break down. In second-order arithmetic, the theory $\ATR_0$, consisting of $\ACA_0$ together with the principle of transfinite recursion, implies $\Delta_1^1$-comprehension and is in fact, much stronger (by Thm~VIII.4.20 in~\cite{Simpson2009:SubsystemsOfSecondOrderArithmetic}, $\ATR_0$ proves that there is a countable coded $\omega$-model of ${\rm ACA}_0$+$\Delta^1_1$-comprehension). In second-order set theory, the relationship between $\ETR$ and $\Delta_1^1$-comprehension is very different. We know that the theory $\GBC+\Delta_1^1$-comprehension cannot imply $\ETR$ because it is $\Pi^1_2$-conservative over $\GBC$ (see Fujimoto~\cite{Fujimoto2012:Classes-and-truths-in-set-theory}, Thm~15). But Sato's result shows that the principle of $\Delta^1_1$-comprehension is not provable from \GBC+\ETR, and in fact, $\GBC+\ETR+\Delta^1_1$-comprehension has strictly stronger consistency strength than $\GBC+\ETR$ alone. Using some of the ideas of the more general \cite[theorem~33]{Sato2014:Relative-predicativity-and-dependent-recursion-in-second-order-set-theory-and-higher-order-theories}, we provide here a streamlined proof of just this result.

\begin{theorem}[Sato]
 The theory $\GBC+\ETR+\Delta^1_1$-comprehension proves the consistency of $\GBC+\ETR$. Consequently, if consistent, the theory $\GBC+\ETR$ does not prove $\Delta^1_1$-comprehension.
\end{theorem}

\begin{proof}
We shall argue in the theory $\GBC+\ETR+\Delta^1_1$-comprehension that there is an encoded collection of classes that forms a model of $\GBC+\ETR$. That is, we shall prove that there is a class $A\of V\times V$, such that the structure $\<V,{\in},\set{A_x\mid x\in V}>$, built from the classes encoded by $A$, is a model of $\GBC+\ETR$, where $A_x=\set{y\mid (x,y)\in A}$ is the $x^{\rm th}$ slice of $A$. By reflecting this situation down to a set, the consistency of $\GBC+\ETR$ follows.

Recall that in $\GBC+\ETR$ we may define the unique first-order truth predicate relative to any class parameter. So in this theory we may freely refer to first-order truth relative to any class parameter. Fix any global well-order $\trianglelt$, and let $A^0$ be an encoded list of all classes that are first-order definable in $\<V,{\in},\trianglelt>$, together with the (unique) solutions to all first-order recursions that are undertaken on well-orders that are definable over this structure. For example, we may place the class defined by formula $\varphi(\cdot,\vec a)$ with set parameters $\vec a$ on the $\<\varphi,\vec a>^{\rm th}$ slice of $A^0$; and similarly, given a first-order recursion $\psi(x,b,F,\vec c)$, defined along a well-founded relation $\leq$ defined by first-order formula $\theta(\cdot,\cdot,\vec e\,)$, we may place the unique solution of the recursion $F$, which exists by \ETR, on the $\<\varphi,\vec c,\theta,\vec e\,>^{\rm th}$ slice of $A^0$. The class $A^0$ exists, since first of all, each slice exists by \ETR, and consequently the whole class $A^0$ is $\Delta^1_1$-definable, on account of the uniqueness of the definable classes and the solutions of the recursions. (A subtle point: one might have hoped to show that $A^0$ exists just in \ETR\ itself, by a recursion that proceeds in parallel through the recursions of each slice, but that way of doing the recursion would actually be $\Delta^1_1$ rather than first-order, because of the need to refer uniformly to the truth of the recursive steps for the partial solutions of the recursion.)

Given a class $A^n\of V\times V$, which we view as encoding the collection of its slices, we similarly define $A^{n+1}$ to be a class encoding all the first-order definable classes that are definable in $\<V,{\in},X>$ for any class $X$ that is coded in $A^n$, as well as all solutions to first-order recursions along a well-founded first-order definable relations, allowing class parameters from amongst the classes coded in $A^n$. If we place the classes into $A^{n+1}$ in the same kind of canonical manner, then $A^{n+1}$ will be $\Delta^1_1$-definable, because the definable classes and solutions to those recursions are unique.

It follows that the unifying class $A=\set{(\<n,x>,y)\mid y\in A^n_x}$, which encodes all of the classes $\set{A^n_x\mid n\in\omega, x\in V}$ that arise in our construction, is also $\Delta^1_1$ definable. We shall complete the argument by proving that the family of classes encoded by $A$ gives rise to a structure $\<V,{\in},\set{A^n_x\mid n\in\omega, x\in V}>$ that is a model of $\GBC+\ETR$. This is almost immediate by the design of the construction. Specifically, for \GBC, any class that is definable from some class parameter $A^n_x$ is added as a slice of the next stage $A^{n+1}$, and so it appears as $A^{n+1}_u$ for some $u$, and is consequently in our family of classes. Similarly, any solution $F$ to a first-order recursion over a first-order definable well-founded relation, defined relative to some classes appearing as slices in $A$, will be added to $A^{n+1}$ after any stage $n$ by which the class parameters have appeared in $A^n$. So \ETR\ will hold for our classes, as desired.

Finally, we note that the existence of a single class encoding a model of $\GBC+\ETR$ implies $\Con(\GBC+\ETR)$, because we may apply the reflection theorem to produce a set-sized model.
\end{proof}

\section{Questions}\label{Section.Questions}

The work of this article suggests numerous questions for further investigation. Can we weaken the assumption of $\Pi^1_1$-comprehension in theorem~\ref{Theorem.KMImpliesOpenDeterminacy} to use only the principle \ETR\ of elementary transfinite recursion over well-founded class relations? If so, it would follow that open determinacy and clopen determinacy for class games are both equivalent over \GBC\ to the principle of transfinite recursion, which would resonate with the corresponding situation in reverse mathematics for games on the natural numbers, where both open determinacy and clopen determinacy are equivalent to the principle of transfinite recursion over $\ACA_0$. But perhaps open determinacy is strictly stronger than clopen determinacy over \GBC. Which strengthening of \GBC\ suffices to prove the meta-$L$ structure we construct in theorem~\ref{Theorem.KMImpliesOpenDeterminacy} is admissible? If this is possible in \GBC\ plus \ETR, then the proper class analogue of the Blass result mentioned earlier might show that open determinacy and clopen determinacy for classes are equivalent over \GBC. Is there a class game analogue of Martin's proof~\cite{Martin1975:BorelDeterminacy} of Borel determinacy? What does it take to prove the class analogue of Borel determinacy for class games? There is a natural concept of class Borel codes, which in $\KM^+$ gives rise to a collection of classes that is the smallest collection of classes containing the open classes and closed under countable unions and complements. Are all such class games determined? If $\kappa$ is an inaccessible cardinal, then the full second-order structure $\<V_\kappa,{\in},V_{\kappa+1}>$ is a model of $\KM^+$ that satisfies Borel determinacy for class games. Is there a proper class analogue of Harvey Friedman's famous proof~\cite{Friedman1971:HigherSetTheoryAndMathematicalPractice} that Borel determinacy requires strength? We have taken up all these questions in current work.

\bibliographystyle{alpha}
\bibliography{MathBiblio,HamkinsBiblio}

\end{document}